\documentclass[reqno]{amsart}
\usepackage{amssymb}
\usepackage{amsfonts}

\title{Self-adjoint elements in the pseudo-unitary group ${\bf U}\left(p,p\right)$}
\author{Sachin Munshi$^*$ and Rongwei Yang}
\thanks{
$^*$ Corresponding author}
\address[Sachin Munshi]{Department of Mathematics and Statistics, SUNY at Albany, Albany, NY 12222, U.S.A.}
\email[Sachin Munshi]{smunshi@albany.edu}

\address[Rongwei Yang]{Department of Mathematics and Statistics, SUNY at Albany, Albany, NY 12222, U.S.A. }
\email[Rongwei Yang]{ryang@albany.edu}

\date{}

\newtheorem{thm}{Theorem}[section]

\newtheorem{lem}[thm]{Lemma}
\newtheorem{cor}[thm]{Corollary}
\theoremstyle{definition}
\newtheorem{defn}[thm]{Definition}
\theoremstyle{plain}
\newtheorem{fact}[thm]{Fact}

\theoremstyle{definition}
\newtheorem{example}[thm]{Example}

\newtheorem{rem}[thm]{Remark}

\def\quotient#1#2{%
    \raise1ex\hbox{$#1$}\Big/\lower1ex\hbox{$#2$}%
}

\begin{document}

\maketitle

\begin{abstract}
The pseudo-unitary group ${\bf U}\left(p,q\right)$ of signature $\left(p,q\right)$ is the group of matrices that preserve the indefinite pseudo-Euclidean metric on the vector space $\mathbb{C}^{p,q}$. The goal of this paper is to describe the set ${\bf U}_{s}\left(p,p\right)$ of Hermitian, or, self-adjoint elements in ${\bf U}\left(p,p\right)$.



\end{abstract}

\section{Introduction}

The pseudo-unitary group ${\bf U}\left(p,q\right)$ is also often called an indefinite unitary group. It is a connected, non-compact Lie group defined by 
\[
{\bf U}\left(p,q\right):=\left\{ M\in{\bf GL}\left(n,\mathbb{C}\right):M^{*}I_{p,q}M=I_{p,q}\right\} ,
\]
 with $I_{p,q}:=\text{diag}\left\{ I_{p},-I_{q}\right\} $, where
$I_{p}$ is the identity matrix of size $p\times p$, and $n=p+q$.  It is the complex analogue of the pseudo-orthogonal
group 
\[
{\bf O}\left(p,q\right):=\left\{ M\in{\bf GL}\left(n,\mathbb{R}\right):M^{T}I_{p,q}M=I_{p,q}\right\} .
\]
 In particular, ${\bf O}\left(1,3\right)$ is the Lorentz group, i.e.
the group of Lorentz transformations associated to the Minkowski metric
$\eta:=\text{diag}\left\{ 1,-1,-1,-1\right\} $ for special relativity.
So then ${\bf U}\left(1,3\right)$ may be viewed as the complex Lorentz
group (\cite{San66}), and ${\bf U}\left(p,q\right)$ as the generalized
complex Lorentz group.

${\bf U}\left(p,q\right)$ also contains certain subgroups that are
of particular interest to researchers in modern math and physics.
One of them is the maximal compact subgroup ${\bf U}\left(p\right)\oplus{\bf U}\left(q\right)$
(\cite{Ner11}). As we shall see, this subgroup is precisely the set
of unitary elements inside the pseudo-unitary group. Another subgroup,
which is extensively studied in the physics literature, is 
\[
{\bf SU}\left(p,q\right):={\bf U}\left(p,q\right)\cap{\bf SL}\left(n,\mathbb{C}\right).
\]
 In particular, ${\bf SU}\left(1,1\right)$ is important for the coherent
states of the harmonic oscillator and the Coulomb problem, in the
study of path integrals in quantum mechanics (\cite{IKP92}). Moreover,
${\bf SU}\left(2,2\right)$ is the symmetry group of twistor space
developed by Roger Penrose, and it is locally isomorphic to ${\bf SO}\left(4,2\right)$
(\cite{ADM17}). For more information about ${\bf U}\left(p,q\right)$, and ${\bf U}\left(p,p\right)$ in particular, we refer readers to \cite{Lou01, Ner11, Por95, Sny97}.

This paper aims to describe the set ${\bf U}_{s}\left(p,q\right)$ of
self-adjoint (Hermitian) elements in ${\bf U}\left(p,q\right)$ for the case $p=q$. By
the Cartan decomposition theorem (\cite{Ner11}, Theorem 3.4), every
matrix $M\in{\bf U}\left(p,q\right)$ is of the form $M=US$, where
$U\in{\bf U}\left(p\right)\oplus{\bf U}\left(q\right)$ and $S$ is
self-adjoint. It is easy to check that $S\in{\bf U}_{s}\left(p,q\right)$,
and hence there is the following factorization: 
\begin{equation}
{\bf U}\left(p,q\right)=\left({\bf U}\left(p\right)\oplus{\bf U}\left(q\right)\right)\cdot{\bf U}_{s}\left(p,q\right).
\end{equation}
 This indicates that the structure of ${\bf U}\left(p,q\right)$ is
closely linked to that of ${\bf U}_{s}\left(p,q\right)$. However,
self-adjoint elements in the classical Lie groups  have not been
explicitly described in any mathematics literature to date.

This paper is organized in the following manner. Section 2 consists
of preliminary information and facts related to ${\bf U}\left(p,q\right)$;
for instance, it leaves invariant particular Hermitian bilinear and
quadratic forms (\cite{Mok89}). For this reason, we are primarily interested
in the Hermitian elements inside the pseudo-unitary group. Therefore,
in Section 3, we introduce the set ${\bf U}_{s}\left(p,q\right)$
of self-adjoint (Hermitian) matrix elements in ${\bf U}\left(p,q\right)$,
and study the structure of these matrices through vector analysis
and unitary transformations. As the main result of this paper, Theorem
3.7 gives a complete description for ${\bf U}_{s}\left(p,p\right)$. In Section 4,
we study ${\bf U}_{s}\left(p,q\right)$ from the perspective of Lie
algebras, and determine the range of the exponential map, again taking
particular care of the case $p=q$. We give our concluding remarks in 
Section 5.

\section{Preliminaries}

Throughout this paper, $I_{p,q}=\text{diag}\left\{ I_{p},-I_{q}\right\} $
denotes the diagonal block matrix of size $n\times n,n=p+q$, where
$I_{p}$ is the identity matrix of size $p\times p$, and ${\bf U}\left(p\right)$
is the group of unitary matrices of size $p\times p$, i.e., $U\in{\bf U}\left(p\right)$
if and only if $U^{*}U=I_{p}$, where $U^{*}$ is the conjugate transpose
(or adjoint) of $U$. We now give the definition of the pseudo-unitary
group of signature $\left(p,q\right)$.
\begin{defn}
The \textit{pseudo-unitary group} of signature $\left(p,q\right)$
is given by 
\begin{equation}
{\bf U}\left(p,q\right)=\left\{ M\in{\bf GL}\left(p+q,\mathbb{C}\right):M^{*}I_{p,q}M=I_{p,q}\right\} ,\label{eq1:-1}
\end{equation}
where $p+q=n$. 
\end{defn}

Note that $I_{p,q}^{2}=I_{n}$ and for $M\in{\bf U}\left(p,q\right)$,
we have $|\det M|=1$. Moreover, ${\bf U}\left(p,0\right)\cong{\bf U}\left(p\right),{\bf U}\left(0,q\right)\cong{\bf U}\left(q\right)$,
and ${\bf U}\left(q,p\right)\cong{\bf U}\left(p,q\right)$. A discussion
on the structure of special types of matrices, including pseudo-unitary
ones, can be found in \cite{MMT03}. For a viewpoint on the pseudo-unitary
group that is more inclined toward pseudo-Euclidean geometry, see
\cite{Ner11}. Without loss of generality, we shall assume $p\leq q$
in this paper.
\begin{defn}
The \textit{special pseudo-unitary group} of signature $\left(p,q\right)$
is given by
\begin{equation}
{\bf SU}\left(p,q\right)=\left\{ M\in{\bf U}\left(p,q\right):\det M=1\right\} .
\end{equation}
\end{defn}

We now mention a few facts about ${\bf U}\left(p,q\right)$ and ${\bf SU}\left(p,q\right)$.
\begin{fact}
For $p+q=n$, we have that
\begin{equation}
{\bf U}\left(p,q\right)\cap{\bf U}\left(n\right)={\bf U}\left(p\right)\oplus{\bf U}\left(q\right).\label{eq2:-1}
\end{equation}
\end{fact}

\begin{proof}
It is easy to see that ${\bf U}\left(p\right)\oplus{\bf U}\left(q\right)\subset{\bf U}\left(p,q\right)\cap{\bf U}\left(n\right)$. 

Writing $M=\begin{pmatrix}M_{11} & M_{12}\\
M_{21} & M_{22}
\end{pmatrix}\in{\bf U}\left(p,q\right)$, where $M_{11}$ and $M_{22}$ are matrices of size $p\times p$
and $q\times q$, respectively, and using the identity $M^{*}I_{p,q}M=I_{p,q}$
we have that
\begin{align}
M_{11}^{*}M_{11}-M_{21}^{*}M_{21} & =I_{p},\nonumber \\
M_{21}^{*}M_{21}-M_{22}^{*}M_{22} & =-I_{q},\label{eq3:-1}\\
M_{11}^{*}M_{12}-M_{12}^{*}M_{22} & =M_{21}^{*}M_{11}-M_{22}^{*}M_{21}=0.\nonumber 
\end{align}
Using identities in (2.5), one verifies that 
\[
M^{-1}=\begin{pmatrix}M_{11}^{*} & -M_{21}^{*}\\
-M_{12}^{*} & M_{22}^{*}
\end{pmatrix}.
\]
If $M$ is in the unitary group ${\bf U}\left(n\right)={\bf U}\left(p+q\right)$,
then $M^{*}=M^{-1}$. Therefore, it follows that $M_{12}=M_{21}=0$,
and hence, $M\in{\bf U}\left(p\right)\oplus{\bf U}\left(q\right)$
by the first two equations in (\ref{eq3:-1}).
\end{proof}
\begin{fact}
For any $z,w\in\mathbb{C}^{p,q}$, both ${\bf U}\left(p,q\right)$
and ${\bf SU}\left(p,q\right)$ leave the following indefinite Hermitian
bilinear and quadratic forms invariant:

\begin{equation}
B_{p,q}\left(z,w\right)=\sum_{i=1}^{p}\bar{z}_{i}w_{i}-\sum_{j=p+1}^{n}\bar{z}_{j}w_{j},
\end{equation}

\begin{equation}
Q_{p,q}\left(z\right)=\sum_{i=1}^{p}|z_{i}|^{2}-\sum_{j=p+1}^{n}|z_{j}|^{2}.
\end{equation}
\end{fact}

\begin{fact}
${\bf SU}\left(1,1\right)\cong{\bf SL}\left(2,\mathbb{R}\right)$
and moreover, $\mathbb{D}\cong{\bf SU}\left(1,1\right)/{\bf U}\left(1\right)$,
where $\mathbb{D}\subset\mathbb{C}$ is the unit disk.
\end{fact}

\begin{proof}
See \cite{KS09,Mok89}.
\end{proof}
\begin{rem}
More quantum dynamical applications for general pseudo-unitary operators
can be found in \cite{Mos04} and references therein.
\end{rem}

\section{Self-adjoint Elements in ${\bf U}\left(p,p\right)$}

Suppose we have a complex Hilbert space $V$ with an inner product
$\langle\cdot,\cdot\rangle$. Let $L:V\rightarrow V$ be a bounded
linear operator. Then its adjoint $L^{*}:V\rightarrow V$ is defined
through the equation
\begin{equation}
\langle Lu,v\rangle=\langle u,L^{*}v\rangle\text{ }\forall u,v\in V.
\end{equation}
$L$ is said to be Hermitian, or self-adjoint, whenever $\langle Lu,v\rangle=\langle u,Lv\rangle$,
which means that $L=L^{*}$. For $V$ finite-dimensional with a given
orthonormal basis, this is the same as saying that the matrix associated
to $L$ is self-adjoint, i.e. equal to its conjugate transpose. We
define ${\bf U}_{s}\left(p,q\right)$ to be the set
\begin{equation}
{\bf U}_{s}\left(p,q\right)=\left\{ M\in{\bf U}\left(p,q\right):M=M^{*}\right\} .\label{eq4:-1}
\end{equation}
 Note that $M\in{\bf U}_{s}\left(p,q\right)$ in (\ref{eq4:-1}) satisfies
the equation $MI_{p,q}M=I_{p,q}$. Moreover, ${\bf U}_{s}\left(p,q\right)$
itself is not a group as the product of two Hermitian matrices is
not necessarily Hermitian, unless the matrices commute. But, as we
shall see, it contains a nontrivial abelian group. The group ${\bf U}\left(p\right)\oplus{\bf U}\left(q\right)$
acts on ${\bf U}_{s}\left(p,q\right)$. Indeed, if $U\in{\bf U}\left(p\right),V\in{\bf U}\left(q\right)$,
then for every $M=\begin{pmatrix}M_{11} & M_{12}\\
M_{12}^{*} & M_{22}
\end{pmatrix}$ in ${\bf U}_{s}\left(p,q\right)$, we have 
\begin{equation}
\hat{M}:=\begin{pmatrix}U^{*} & 0\\
0 & V^{*}
\end{pmatrix}\begin{pmatrix}M_{11} & M_{12}\\
M_{12}^{*} & M_{22}
\end{pmatrix}\begin{pmatrix}U & 0\\
0 & V
\end{pmatrix}=\begin{pmatrix}U^{*}M_{11}U & U^{*}M_{12}V\\
V^{*}M_{12}^{*}U & V^{*}M_{22}V
\end{pmatrix},
\end{equation}
which is obviously self-adjoint. Moreover, one computes easily that
\begin{align*}
\hat{M}I_{p,q}\hat{M} & =\text{diag}\left\{ U^{*},V^{*}\right\} M\text{diag}\left\{ U,V\right\} I_{p,q}\text{diag}\left\{ U^{*},V^{*}\right\} M\text{diag}\left\{ U,V\right\} \\
 & =\text{diag}\left\{ U^{*},V^{*}\right\} MI_{p,q}M\text{diag}\left\{ U,V\right\} \\
 & =\text{diag}\left\{ U^{*},V^{*}\right\} I_{p,q}\text{diag}\left\{ U,V\right\} =I_{p,q}.
\end{align*}

Moreover, if $M\in{\bf U}_{s}\left(p,q\right)$, then clearly $-M\in{\bf U}_{s}\left(p,q\right)$,
hence the group $\left\{ \pm1\right\} $ acts on ${\bf U}_{s}\left(p,q\right)$
as well. These observations justify the following. 
\begin{defn}
Two elements $M_{1},M_{2}\in{\bf U}\left(p,q\right)$ are said to
be \textit{equivalent}, and denoted by $M_{1}\sim M_{2}$, if $M_{1}=\pm M_{2}$
or $M_{1}=Q^{*}M_{2}Q$ for some $Q\in{\bf U}\left(p\right)\oplus{\bf U}\left(q\right)$.
\end{defn}

To gain a better sense of the general case, we first look at the example
of ${\bf U}_{s}\left(1,1\right)$. 
\begin{example}
\noindent $p=q=1$. Assume $M=\begin{pmatrix}m_{11} & m_{12}\\
\bar{m}_{12} & m_{22}
\end{pmatrix}\in{\bf U}_{s}\left(1,1\right)$. Then by (2.5), we have  \begin{equation}   
\begin{aligned}     
m_{11}^{2} - |m_{12}|^{2} & =  1, \\     
|m_{12}|^{2} - m_{22}^{2} & =  -1, \\     
m_{12}(m_{11} - m_{22}) & = 0.   
\end{aligned} 
\end{equation}Clearly, if $m_{12}=0$, then $M=\text{diag}\left\{ \pm1,\pm1\right\} $.
If $m_{12}\neq0$, then $m_{11}=m_{22}$, and we can write 
\[
M=\begin{pmatrix}c & s\\
\bar{s} & c
\end{pmatrix},
\]
 where $s\neq0$ and $c\in\mathbb{R}$ with $c^{2}=1+\left|s\right|^{2}$.
Up to an action by ${\bf U}\left(1\right)\oplus{\bf U}\left(1\right)$
as in (3.10), we may assume $s$ is real and $c\geq1$. Letting $c=\cosh t$
and $s=\sinh t$, we have 
\begin{equation}
M_{t}:=\begin{pmatrix}\cosh t & \sinh t\\
\sinh t & \cosh t
\end{pmatrix},\text{ \text{ \text{ }}}t\in\mathbb{R}.
\end{equation}
Using identities of $\cosh$ and $\sinh$ functions, it is easy to
check that 
\[
M_{t}M_{t^{'}}=M_{t+t^{'}}.
\]
 Hence, $G:=\left\{ M_{t}:t\in\mathbb{R}\right\} $ is a nontrivial
abelian one-parameter group isomorphic to $\left(\mathbb{R},+\right)$.
$G$ can be viewed as a hyperbolic group that is studied in \cite{Bar47,Wol88}.
It is easy to check that $M_{-t}=I_{1,1}M_{t}I_{1,1}$, hence $M_{-t}\sim M_{t}$.
Let $G_{+}=\left\{ M_{t}:t\geq0\right\} $. Then $G_{+}$ is an abelian
semigroup. Moreover, since the trace $\text{Tr}\left(M_{t}\right)=2\cosh t$
is a strictly increasing function on the interval $[0,\infty)$, we
see that no two distinct elements in $G_{+}$ are equivalent. Therefore,
we have $\quotient{\mathbf{U}_{s}(1,1)}{\sim} = G_{+} \cup \{I_{1,1}\}$.
Further, if we let  \begin{align*} {\bf U}_{s+}\left(1,1\right) :=\left\{ M\in{\bf U}_{s}\left(1,1\right): \text{Tr}(M)\geq 2 \right\}, 
\end{align*}then the above observations indicate that ${\bf U}_{s+}\left(1,1\right)$
is the connected component of ${\bf U}_{s}\left(1,1\right)$ that
contains the identity matrix $I_{2}$. Moreover, we have the disjoint
decomposition 
\[
{\bf U}_{s}\left(1,1\right)={\bf U}_{s+}\left(1,1\right)\cup-{\bf U}_{s+}\left(1,1\right)\cup\left\{ \pm I_{1,1}\right\} .
\]
 Furthermore, it is not hard to see that ${\bf SU}_{s}\left(1,1\right)$
is also invariant under the multiplication by $\left\{ \pm1\right\} $
and the action by ${\bf U}\left(1\right)\oplus{\bf U}\left(1\right)$
as defined in (3.10). Hence, the quotient $\quotient{{\bf SU}_{s}\left(1,1\right)}{\sim}$
is well-defined. Since $\det I_{1,1}=-1$ and $\det M=1$ for all
$M\in G_{+}$, we see that $\quotient{{\bf SU}_{s}\left(1,1\right)}{\sim}=G_{+}$.
\end{example}

Now we move on to the general case. Consider an arbitrary $M\in{\bf U}_{s}\left(p,q\right)$.
The fact $MI_{p,q}M=I_{p,q}$ is equivalent to the following equation:
\begin{equation}
\left(M-I_{p,q}\right)I_{p,q}\left(M+I_{p,q}\right)=0.
\end{equation}
 The Sylvester rank inequality then implies the following. 
\begin{cor}
If $M\in{\bf U}_{s}\left(p,q\right)$, then $\textnormal{rank}(M-I_{p,q})+ \textnormal{rank}(M+I_{p,q}) \leq p+q$.
\end{cor}

Without loss of generality, we shall assume $\text{rank}\left(M-I_{p,q}\right)\geq\text{rank}\left(M+I_{p,q}\right)$,
as otherwise we may work with $-M$. Then $k:=\text{rank}\left(M+I_{p,q}\right)\leq q$.
Since $M+I_{p,q}$ is Hermitian, using spectral decomposition as in
\cite{Lay99}, we may write
\begin{equation}
M+I_{p,q}=\sum_{i=1}^{k}\lambda_{i}z_{i}z_{i}^{*},\label{eqa:}
\end{equation}
where $\lambda_{1},\lambda_{2},\text{\dots},\lambda_{k}$ are the
nonzero eigenvalues of $M+I_{p,q}$ whose eigenvectors $z_{i}$ are
assumed orthonormal throughout this paper, i.e. $\left\langle z_{i},z_{j}\right\rangle =z_{i}^{*}z_{j}=\delta_{ij}$.\footnote{in the standard Euclidean inner product on $\mathbb{C}^{n}$}
With slight algebraic tweaking we see that $MI_{p,q}M=I_{p,q}$ if
and only if
\begin{equation}
\sum_{i,j=1}^{k}\lambda_{i}\lambda_{j}z_{i}\left(z_{i}^{*}I_{p,q}z_{j}\right)z_{j}^{*}-2\sum_{j=1}^{k}\lambda_{j}z_{j}z_{j}^{*}=0.\label{eq1:}
\end{equation}
 Since $\left\{ z_{j}:1\leq j\leq k\right\} $ is an orthonormal set,
applying the standard inner product on $\mathbb{C}^{n}$ with $z_{j}$
from the right to (\ref{eq1:}), we see that (\ref{eq1:}) is equivalent
to
\begin{equation}
\sum_{i=1}^{k}\lambda_{i}\lambda_{j}z_{i}\left(z_{i}^{*}I_{p,q}z_{j}\right)-2\lambda_{j}z_{j}=0\text{, \text{ }}\forall1\leq j\leq k,
\end{equation}
 and this is true if and only if
\begin{equation}
\begin{cases}
z_{i}^{*}I_{p,q}z_{j}=0; & i\neq j,\\
\lambda_{j}^{2}z_{j}^{*}I_{p,q}z_{j}-2\lambda_{j}=0; & i=j.
\end{cases}\label{eq2:}
\end{equation}
Note that $z_{i}^{*}I_{p,q}z_{j}$ in (\ref{eq2:}), regardless of
whether $i=j$ or not, is a scalar, not a matrix. From the
second equation in (\ref{eq2:}), the nonzero eigenvalues of $M+I_{p,q}$
are $\lambda_{j}=\frac{2}{z_{j}^{*}I_{p,q}z_{j}}$.

If we write vectors $z_{i}=\begin{pmatrix}z_{i}^{+}\\
z_{i}^{-}
\end{pmatrix}$, where $z_{i}^{+}\in\mathbb{C}^{p},z_{i}^{-}\in\mathbb{C}^{q}$,
then for $i\neq j$, from the first equation in (\ref{eq2:}), we
have that  \begin{equation}   
\begin{aligned}
\langle    
\begin{pmatrix}z_{i}^{+} \\
 -z_{i}^{-}
\end{pmatrix},&
\begin{pmatrix}z_{j}^{+}\\
z_{j}^{-}
\end{pmatrix} \rangle =  \langle z_{i}^{+}, z_{j}^{+} \rangle - \langle z_{i}^{-}, z_{j}^{-} \rangle = 0, \\
\langle
\begin{pmatrix}z_{i}^{+} \\
z_{i}^{-}
\end{pmatrix},&
\begin{pmatrix}z_{j}^{+}\\
z_{j}^{-}
\end{pmatrix} \rangle =  \langle z_{i}^{+}, z_{j}^{+} \rangle + \langle z_{i}^{-}, z_{j}^{-} \rangle = 0,      
\end{aligned} \label{eq3:}
\end{equation}where $\left\langle \text{ ,}\text{ }\right\rangle $ stands for the
Euclidean inner product on $\mathbb{C}^{n}$. It follows that $\langle z_{i}^{+},z_{j}^{+}\rangle=0=\langle z_{i}^{-},z_{j}^{-}\rangle$,
so $z_{i}^{+},z_{j}^{+}$ and $z_{i}^{-},z_{j}^{-}$ are orthogonal
pairs for $i\neq j$. This means that $\left\{ z_{j}^{+}\right\} _{j=1}^{k},\left\{ z_{j}^{-}\right\} _{j=1}^{k}$
are orthogonal sets in $\mathbb{C}^{p},\mathbb{C}^{q}$, respectively.
Note that it is possible that for some $j$ we have $z_{j}^{+}=0$
or $z_{j}^{-}=0$. We summarize the above observations as follows. 
\begin{lem}
$M\in{\bf U}_{s}\left(p,q\right)$ if and only if there are orthogonal
sets $\left\{ z_{j}^{+}\right\} _{j=1}^{k}$ and $\left\{ z_{j}^{-}\right\} _{j=1}^{k}$
in $\mathbb{C}^{p}$ and $\mathbb{C}^{q}$, respectively, with $k\leq q$,
$\left\Vert z_{j}^{+}\right\Vert ^{2}+\left\Vert z_{j}^{-}\right\Vert ^{2}=1$,
and $\left\Vert z_{j}^{+}\right\Vert ^{2}\neq\left\Vert z_{j}^{-}\right\Vert ^{2}$
such that either $M$ or $-M$ is of the form
\[
\sum_{j=1}^{k}\lambda_{j}z_{j}z_{j}^{*}-I_{p,q},
\]
 where $z_{j}=\begin{pmatrix}z_{j}^{+}\\
z_{j}^{-}
\end{pmatrix}$ and $\lambda_{j}=\frac{2}{\left\Vert z_{j}^{+}\right\Vert ^{2}-\left\Vert z_{j}^{-}\right\Vert ^{2}}$.
\end{lem}

The following corollary is immediate.
\begin{cor}
If $M\in{\bf U}_{s}\left(p,q\right)$ and $\lambda\neq0$ is an eigenvalue
of $M+I_{p,q}$, then $|\lambda|\geq2$.
\end{cor}

\begin{example}
We now look at the case $p=1,q=2$. Assume $M\in{\bf U}_{s}\left(1,2\right)$.
Then there are four cases. If $\text{rank}\left(M+I_{1,2}\right)=0$,
then clearly $M=-I_{1,2}$. Next, $\text{rank}\left(M+I_{1,2}\right)=1$
if and only if 
\begin{equation}
M=\lambda zz^{*}-I_{1,2},
\end{equation}
 where $z=\begin{pmatrix}z_{1}\\
z_{2}\\
z_{3}
\end{pmatrix}$ such that $\left\Vert z\right\Vert =1$, and
\begin{equation}
\lambda z^{*}I_{1,2}z=\lambda\left(\left|z_{1}\right|^{2}-\left|z_{2}\right|^{2}-\left|z_{3}\right|^{2}\right)=2.
\end{equation}
 If $\text{rank}\left(M+I_{1,2}\right)=2$, then by Corollary 3.3,
$\text{rank}\left(M-I_{1,2}\right)\leq1$, and hence, $-M$ fits into
the above discussion. Finally, if $\text{rank}\left(M+I_{1,2}\right)=3$,
then we have that $\text{rank}\left(M-I_{1,2}\right)=0$ , and hence,
$M=I_{1,2}$.\medskip{}

Now we are in the position to state and prove the main theorem of
this paper.
\end{example}

\begin{thm}
For every integer $p\geq1$ we have that\[\quotient{{\bf U}_{s}\left(p,p\right)}{\sim} =\underbrace{(G_{+}\cup\{I_{1,1}\}) \oplus (G_{+}\cup\{I_{1,1}\}) \oplus \cdots \oplus (G_{+}\cup\{I_{1,1}\})}_{p}.\] 
\end{thm}

\begin{proof}
We show $M\in{\bf U}_{s}\left(p,p\right)$ if and only if there are
$M_{j}\in{\bf U}_{s}\left(1,1\right),1\leq j\leq p$, such that 
\[
M\sim M_{1}\oplus M_{2}\oplus\cdots\oplus M_{p}.
\]
 The theorem then follows from the discussion in Example 3.2.

In Lemma 3.4, since $\left\{ z_{j}^{+}\right\} _{j=1}^{k}$ and $\left\{ z_{j}^{-}\right\} _{j=1}^{k}$
are orthogonal sets, there exist orthonormal bases $\left\{ e_{j}^{'}\right\} _{j=1}^{p}$
and $\left\{ f_{j}^{'}\right\} _{j=1}^{q}$ for $\mathbb{C}^{p}$
and $\mathbb{C}^{q}$, respectively, such that 
\[
z_{j}^{+}=\left\Vert z_{j}^{+}\right\Vert e_{j}^{'},\text{ \text{ \text{ }}}z_{j}^{-}=\left\Vert z_{j}^{-}\right\Vert f_{j}^{'},\text{ \text{ \text{ }}}1\leq j\leq k.
\]
 Then there exist $U\in{\bf U}\left(p\right)$ and $V\in{\bf U}\left(q\right)$
such that $Ue_{j}^{'}=e_{j}$ and $Vf_{j}^{'}=f_{j}$, where 
\[
e_{j}=\begin{pmatrix}0\\
\vdots\\
1\\
\vdots\\
0
\end{pmatrix}\in\mathbb{C}^{p},f_{j}=\begin{pmatrix}0\\
\vdots\\
1\\
\vdots\\
0
\end{pmatrix}\in\mathbb{C}^{q},
\]
with 1 in the $j$-th position. Setting $\alpha_{j}=\left\Vert z_{j}^{+}\right\Vert $
and $\beta_{j}=\left\Vert z_{j}^{-}\right\Vert $ , we have $\alpha_{j}^{2}+\beta_{j}^{2}=z_{j}^{*}z_{j}=1$.
This allows us to consider $M+I_{p,q}$ under the equivalence $\sim$,
i.e.
\begin{equation}
M+I_{p,q}\sim\begin{pmatrix}U & 0\\
0 & V
\end{pmatrix}\left(M+I_{p,q}\right)\begin{pmatrix}U^{*} & 0\\
0 & V^{*}
\end{pmatrix}.\label{eq4:}
\end{equation}
In view of Lemma 3.4, expanding out the RHS of (\ref{eq4:}), we have:\begin{equation}
\begin{pmatrix}U & 0\\ 
0 & V \end{pmatrix}
\left(M+I_{p,q}\right)
\begin{pmatrix}U^{*} & 0\\ 
0 & V^{*} \end{pmatrix} =
\sum_{j=1}^{k} \lambda_{j}
\begin{pmatrix} \underbrace{\alpha_{j}^{2}e_{j}e_{j}^{*}}_{p \times p} & \underbrace{\alpha_{j}\beta_{j}e_{j}f_{j}^{*}}_{p \times q}\\ 
\underbrace{\alpha_{j}\beta_{j}f_{j}e_{j}^{*}}_{q \times p} & \underbrace{\beta_{j}^{2}f_{j}f_{j}^{*}}_{q \times q} \end{pmatrix}.
\label{eq5:} 
\end{equation}

Now we focus on the case $p=q$. Then $e_{j}=f_{j}$ so that, in view
of (\ref{eq5:}), we have equations  \begin{align}		 
\begin{pmatrix}U & 0\\  
0 & V \end{pmatrix} \left(M+I_{p,p}\right) 
\begin{pmatrix}U^{*} & 0\\  0 & V^{*} \end{pmatrix} &= \sum_{j=1}^{k} \lambda_{j}  
\begin{pmatrix} \alpha_{j}^{2} & \alpha_{j}\beta_{j}\\ \alpha_{j}\beta_{j} & \beta_{j}^{2} \end{pmatrix} \otimes e_{j}e_{j}^{*},  \label{eq13:}\\ \begin{pmatrix}U & 0\\  0 & V \end{pmatrix} I_{p,p} \begin{pmatrix}U^{*} & 0\\  0 & V^{*} \end{pmatrix} &=I_{p,p}=\sum_{j=1}^{p}  \begin{pmatrix} 1 & 0\\ 0 & -1 \end{pmatrix} \otimes e_{j}e_{j}^{*}, \label{eq13':} \end{align}where $U,V\in{\bf U}\left(p\right)$. Note that the index for the
summation in (3.23) is up to $k$, and that in (3.24) it is up to
$p$. Since $P_{j}:=e_{j}e_{j}^{*}$ is the rank 1 projection onto
$\mathbb{C}e_{j}$, we have
\[
\sum_{j=1}^{p}P_{j}=I_{p},\text{ \text{ \text{ }}}P_{i}P_{j}=0\text{ for\text{ }}i\neq j.
\]
 Letting $K_{j}=\begin{pmatrix}\alpha_{j}^{2} & \alpha_{j}\beta_{j}\\
\alpha_{j}\beta_{j} & \beta_{j}^{2}
\end{pmatrix}$, and subtracting (\ref{eq13':}) from (\ref{eq13:}), we have  \begin{equation} \begin{pmatrix}U & 0\\  0 & V \end{pmatrix} M \begin{pmatrix}U^{*} & 0\\  0 & V^{*} \end{pmatrix}=\sum_{j=1}^k\left(\lambda_{j}K_{j}-I_{1,1}\right)\otimes P_j-\sum_{j=k+1}^pI_{1,1}\otimes P_j.\label{eq15:} \end{equation}
Let
\[
K=\begin{pmatrix}\alpha^{2} & \alpha\beta\\
\alpha\beta & \beta^{2}
\end{pmatrix},\text{ \text{ \text{ }}}\lambda=\frac{2}{\alpha^{2}-\beta^{2}},
\]
 where $\alpha,\beta\in\mathbb{R}$ satisfy $\alpha^{2}+\beta^{2}=1$
and $\alpha\neq\pm\beta$. Setting  \begin{align*} m_{11}=\frac{1}{2\alpha^{2}-1},\\ 
m_{12}=\frac{2\alpha\beta}{\alpha^{2}-\beta^{2}},\\  
m_{22}=\frac{1}{1-2\beta^{2}}, \end{align*}one verifies that the matrix $\lambda K-I_{1,1}=(m_{ij})$ satisfies
(3.11) and hence is in ${\bf U}_{s}\left(1,1\right)$. It then follows
from (3.25) that 
\[
M\sim\sum_{j=1}^{p}M_{j}\otimes P_{j},
\]
for some $M_{j}\in{\bf U}_{s}\left(1,1\right),1\leq j\leq p$. 

Conversely, if $M_{j}\in{\bf U}_{s}\left(1,1\right),1\leq j\leq p$,
and 
\[
M=\sum_{j=1}^{p}M_{j}\otimes P_{j},
\]
 then using (3.24), it is easy to check that $MI_{p,p}M=I_{p,p}$,
i.e., $M\in{\bf U}_{s}\left(p,p\right)$. The theorem thus follows
from the discussion in Example 3.2.
\end{proof}
Since $G$ is an abelian one-parameter group in ${\bf U}_{s}\left(1,1\right)$,
the above observations also yield
\begin{cor}
For every integer $p\geq1$, the set ${\bf U}_{s}\left(p,p\right)$
contains an abelian $p$-parameter group that is isomorphic to $\underbrace{G\oplus G\oplus \cdots \oplus G}_{p}.$
\end{cor}

Like that in Example 3.2, it is not hard to see that ${\bf SU}_{s}\left(2,2\right)$
is also invariant under the multiplication by $\left\{ \pm1\right\} $
and the action by ${\bf U}\left(2\right)\oplus{\bf U}\left(2\right)$
as defined in (3.10). Hence, the next corollary follows readily from
Theorem 3.7.
\begin{cor}
 $\quotient{{\bf SU}_{s}\left(2,2\right)}{\sim}=\left(G_+\oplus G_+\right)\cup \{I_{1,1}\oplus I_{1,1}\}$.
\end{cor}

\section{${\bf U}_{s}\left(p,q\right)$, Lie Algebras, and Exponential Map}

In this section we take a look at ${\bf U}_{s}\left(p,q\right)$ from
a Lie algebra point of view. Consider the Lie algebra $\frak{u}\left(p,q\right)$
of the Lie group ${\bf U}\left(p,q\right)$ and the associated exponential
map $\exp:\frak{u}\left(p,q\right)\to{\bf U}\left(p,q\right)$. For
every $T\in\frak{u}\left(p,q\right)$, we consider the one-parameter
subgroup
\begin{equation}
M\left(t\right)=\exp\left(tT\right),\text{ }T=\begin{pmatrix}T_{11} & T_{12}\\
T_{21} & T_{22}
\end{pmatrix},\text{ }t\in\mathbb{R}.\label{eq7:}
\end{equation}
Since we need $M\left(t\right)\in{\bf U}\left(p,q\right)$, we have
that \begin{equation}
 \exp(tT^{*})I_{p,q} \exp(tT) = I_{p,q}. 
 \label{eq8:}
\end{equation}Differentiating (\ref{eq8:}) with respect to $t$, we have 
\[
T^{*}\exp\left(tT^{*}\right)I_{p,q}\exp\left(tT\right)+\exp\left(tT^{*}\right)I_{p,q}T\exp\left(tT\right)=0.
\]
Evaluating the above at $t=0$, we obtain the linear equation \begin{equation}
T^{*}I_{p,q} + I_{p,q}T = 0.
 \label{eq9:}
\end{equation}This shows that the Lie algebra 
\[
\frak{u}\left(p,q\right)=\left\{ T\in{\bf M}_{n}\left(\mathbb{C}\right):T^{*}I_{p,q}=-I_{p,q}T\right\} .
\]
 Clearly, $\exp\left(T\right)$ is self-adjoint if and only if $T$
is self-adjoint. Hence, we set 
\[
\frak{u}_{s}\left(p,q\right):=\left\{ T\in\frak{u}\left(p,q\right):T=T^{*}\right\} .
\]
 Then $\exp\left(T\right)\in{\bf U}_{s}\left(p,q\right)$ if and only
if $T\in\frak{u}_{s}\left(p,q\right)$. Substituting $T$ from (4.26)
into (4.28) and solving $T=T^{*}$, we see that $T_{11}=T_{22}=0$
and $T_{21}=T_{12}^{*}$. Letting $\hat{T}=T_{12}$, we have the following.
\begin{cor}
\noindent Let $T\in{\bf M}_{n}\left(\mathbb{C}\right)$. Then the
following are equivalent.

\noindent (a) $\exp\left(T\right)\in{\bf U}_{s}\left(p,q\right)$.

\noindent (b) $T\in\frak{u}_{s}\left(p,q\right)$.

\noindent (c) $T$ is of the form $\begin{pmatrix}0 & \hat{T}\\
\hat{T}^{*} & 0
\end{pmatrix},\text{ }\hat{T}\in{\bf M}_{pq}\left(\mathbb{C}\right)$, where ${\bf M}_{pq}\left(\mathbb{C}\right)$ is the algebra of $p\times q$
complex matrices.
\end{cor}

Corollary 4.1(c), in particular, implies that $\frak{u}_{s}\left(p,q\right)$
is of complex dimension $pq$. Since $\exp$ is a homeomorphism from
an open neighborhood of $0\in\frak{u}_{s}\left(p,q\right)$ to an
open neighborhood of $I\in{\bf U}_{s}\left(p,q\right)$, the generic
dimension of ${\bf U}_{s}\left(p,q\right)$ is $pq$ as well. By direct
computation we have 
\begin{equation}
M=\exp\left(T\right)=\begin{pmatrix}\cosh|\hat{T}^{*}| & U\sinh|\hat{T}|\\
\sinh|\hat{T}|U^{*} & \cosh|\hat{T}|
\end{pmatrix},
\end{equation}
 where $|A|$ stands for $\sqrt{A^{*}A}$ for any matrix $A$ and
$T=U|T|$ is the polar decomposition of $T$. In view of Example 3.2,
where we considered the case $p=q=1$, we have 
\begin{equation}
T=\begin{pmatrix}0 & \xi\\
\bar{\xi} & 0
\end{pmatrix}\in\frak{u}_{s}\left(1,1\right),\text{ \text{ \text{ }}}\xi=|\xi|e^{i\theta},\text{ }0\leq\theta\leq2\pi,
\end{equation}
 and using (4.29), we obtain 
\begin{equation}
\exp\left(T\right)=M=\begin{pmatrix}\cosh|\xi| & e^{i\theta}\sinh|\xi|\\
e^{-i\theta}\sinh|\xi| & \cosh|\xi|
\end{pmatrix}.
\end{equation}

This calculation and Example 3.2 lead to the following.
\begin{cor}
$\exp\left(\frak{u}_{s}\left(1,1\right)\right)={\bf U}_{s+}\left(1,1\right)$.
Moreover, it follows that $$\quotient{\exp (\mathfrak{u}_{s}(1,1))}{\sim} =G_{+}.$$
\end{cor}

For the general case, using the singular value decomposition, for
$\hat{T}\in{\bf M}_{pq}\left(\mathbb{C}\right)$, there exist unitary
$U\in{\bf M}_{p}\left(\mathbb{C}\right)$ and $V\in{\bf M}_{q}\left(\mathbb{C}\right)$
such that 
\[
U^{*}\hat{T}V=\text{diag}\left\{ s_{1},s_{2},\text{\dots,}s_{p}\right\} ,
\]
 where $s_{j}\in\mathbb{R}$ are singular numbers of $T$. Note that
it is possible $s_{j}=0$ for some $j$. Then, we have 
\begin{align*}
\begin{pmatrix}U^{*} & 0\\
0 & V^{*}
\end{pmatrix}T\begin{pmatrix}U & 0\\
0 & V
\end{pmatrix} & =\begin{pmatrix}0 & U^{*}\hat{T}V\\
V^{*}\hat{T}^{*}U & 0
\end{pmatrix}\\
 & =\begin{pmatrix}0 & \text{diag}\left\{ s_{1},\text{\dots,}s_{p}\right\} \\
\text{diag}\left\{ s_{1},\text{\dots,}s_{p}\right\} ^{*} & 0
\end{pmatrix}.
\end{align*}

So in the case $p=q$, using the $P_{j}$s from (3.24), we have that
\[
T\sim\sum_{j=1}^{p}\begin{pmatrix}0 & s_{j}\\
s_{j} & 0
\end{pmatrix}\otimes P_{j}.
\]
 Exponentiating the above equivalence identity gives us
\[
\exp\left(T\right)\sim\prod_{j=1}^{p}\begin{pmatrix}\cosh s_{j} & \sinh s_{j}\\
\sinh s_{j} & \cosh s_{j}
\end{pmatrix}\otimes P_{j}.
\]
 This gives another view on Theorem 3.7, although the map $\exp$
is not surjective. The next corollary follows from Corollary 4.2.
\begin{cor}
For every integer $p\geq1$, we have $$\quotient{\exp(\mathfrak{u}_{s}(p,p))}{\sim} = \underbrace{G_{+} \oplus \cdots \oplus G_{+}}_{p}.$$ 
\end{cor}

\section{Concluding Remarks}

The pseudo-unitary group ${\bf U}\left(p,q\right)$ is a subject of
many studies as in \cite{IKP92, MMT03, Mok89, Mos04, Ner11}. However, 
the structure of Hermitian (or self-adjoint) matrix elements inside this 
group is not known to have been thoroughly investigated. Using techniques 
from linear algebra, vector analysis, and Lie theory, Theorem 3.7 and Corollary 
4.3 give an explicit picture of ${\bf U}_{s}\left(p,p\right)$. The proof of 
Theorem 3.7, which relies on (3.23) and (3.24), doesn't seem to have an easy 
generalization to the $p\neq q$ case. 

While this paper has not developed a full-scale classification scheme
for describing all of ${\bf U}_{s}\left(p,q\right)$, we are encouraged
and motivated by the results obtained so far. As a possible application of this paper,
self-adjoint elements in ${\bf U}\left(p,p\right)$ could be useful
in the study of Clifford algebras which are widely used in the physics literature (see  
\cite{Lou01, Por95, Shi10, Sny97}). For instance, the group ${\bf CU}\left(p,q\right)$ of $c$-unitary elements in the complex Clifford algebra $C\ell (p,q)$ is isomorphic to the pseudo-unitary group
${\bf U}\left(2^{m-1},2^{m-1}\right)$ if $q\neq 0$ and $p+q=2m$ (\cite{Sny97}).  More recently, links have
been established between Clifford algebras of signature $(p,q)$ associated with 
${\bf O} \left(p,q\right)$ and certain equations at the heart of modern physics such as the
Yang-Mills equations and the Proca equation (\cite{MS16, Shi18}). 

Although technically more challenging, extending the study of this paper to self-adjoint elements in ${\bf U} \left(p,q\right)$ is a promising next step. Furthermore, describing self-adjoint elements
in some other classical groups such as the indefinite symplectic group ${\bf Sp}\left(p,q\right)$ and the spin group ${\bf Spin}\left(p,q\right)$ is also an appealing subject of study. What is even more interesting, however, is to find applications of this study to Clifford algebras, group representation theory, geometry, or even quantum physics.

\section{Acknowledgments}

The authors would like to thank Professor Oleg Lunin for valuable comments on the initial draft of this paper. The first author is also grateful
to the support from the Department of Mathematics and Statistics in SUNY at Albany for
providing him an opportunity to pursue his research interests.

\end{document}